\documentclass[a4paper,11pt,UKenglish]{article}
\usepackage[top=3.2cm, bottom=3.2cm, left=3.45cm, right=3.45cm]{geometry}
\usepackage[T1]{fontenc}
\usepackage[utf8]{inputenc}
\usepackage{lmodern}
\usepackage{amsmath,amssymb,amsthm,mathtools}
\usepackage{tikz}
\usetikzlibrary{patterns.meta, matrix}
\tikzset{NE-lines/.style={
  pattern={Lines[angle=45,distance=4pt]}, pattern color=black!50
}}
\usetikzlibrary{matrix,arrows.meta, positioning}
\usepackage{booktabs}
\usepackage{microtype}
\usepackage[colorlinks=true,linkcolor=blue,citecolor=blue,urlcolor=blue]{hyperref}

\theoremstyle{plain}
\newtheorem{theorem}{Theorem}[section]
\newtheorem{proposition}[theorem]{Proposition}
\newtheorem{corollary}[theorem]{Corollary}
\newtheorem{lemma}[theorem]{Lemma}
\theoremstyle{remark}

\newtheorem{example}[theorem]{Example}
\newtheorem{definition}[theorem]{Definition}

\newcommand{\Cay}{\mathrm{Cay}}

\newcommand{\rgf}{\textnormal{\scshape rgf}}
\DeclareMathOperator{\st}{\mathrm{st}}
\DeclareMathOperator{\img}{im}
\newcommand{\stirling}[2]{\genfrac\{\}{0pt}{}{#1}{#2}}

\title{Simple Cayley permutations}
\author{Giulio Cerbai\\
  \normalsize Independent researcher\\
  \normalsize \texttt{giuliocerbai14@gmail.com}
  \and
  Anders Claesson\\
  \normalsize Department of Mathematics\\
  \normalsize University of Iceland\\
  \normalsize \texttt{akc@hi.is}}
\date{21 October 2026}

\begin{document}

\maketitle

\begin{abstract}
  We propose a notion of simplicity for Cayley permutations that is
  compatible with inflation and agrees, under a natural relational
  encoding, with indecomposability in the sense of Schmerl and
  Trotter.  We establish a substitution decomposition analogous to
  that of permutations, in which the components are modules of the
  associated relational structure.  We use it to enumerate simple
  Cayley permutations as well as simple restricted growth functions.
  \medskip

\end{abstract}

\thispagestyle{empty}
\section{Introduction}\label{intro}

A \emph{Cayley permutation} is either empty or a word of positive
integers that contains every number between one and its maximum
value. In other words, a Cayley permutation is an endofunction $\omega$ on
$[n]=\{1,2,\ldots,n\}$ such that $\img(\omega)=[m]$ for some $m\leq n$.
We write $\Cay_n$ for the set of Cayley permutations of length $n$ and
$\Cay=\bigcup_{n\geq 0}\Cay_n$. In addition, we write
$\Cay_+=\bigcup_{n\geq1}\Cay_n$ for the set of nonempty Cayley permutations.
There is a simple one-to-one correspondence
between ballots (ordered set partitions) and Cayley permutations:
a ballot $\varpi$ corresponds to the Cayley
permutation $\omega$ whose $i$th letter equals the unique index $j$
such that $i$ belongs to the $j$th block of $\varpi$.
Cayley permutations may also be seen as representatives for
equivalence classes of words modulo order isomorphism.
For these and many other reasons, recent papers on Cayley permutations
have explored their combinatorial and enumerative
properties~\cite{BeanBO26,CCtop23,CCcaypol24,CCcaypol26,CCfundig26},
as well as pattern avoidance~\cite{CeSortCay,CeModasc,CCEGcayspec}.

The main objective of this paper is to define simple
Cayley permutations and develop the surrounding theory.
The notion of simplicity on (standard) permutations was introduced
by Albert and Atkinson~\cite{AlbAtk05}: a permutation is
\emph{simple} if it has no proper interval of length at least two, where
an interval is a set of consecutive positions whose
corresponding values also form a contiguous set of integers.
Every nonempty permutation of length $n$ has trivial intervals of
length 0 (the empty set), length 1 (single entries), and length $n$
(the whole permutation).
A permutation is simple if these are its only intervals.
Simple permutations play an important role
in the study of permutation classes
(see the survey by Brignall~\cite{Bri10} for a broad introduction).
Albert and Atkinson~\cite{AlbAtk05} showed that every permutation
may be written as the inflation of a unique simple permutation.
This allowed Albert, Atkinson, and Klazar~\cite{AlbAtkKla03} to
enumerate simple permutations in terms of the compositional inverse
of $f(x)=\sum_{n\ge 1}n!x^n$.

Here, we propose two natural substitution decompositions
for Cayley permutations.
One uses simple quotients and modules of an associated relational structure
as parts; the other uses semisimple quotients and Hertzsprung intervals
or singletons as parts.
We begin with the latter decomposition, defining semisimplicity using
Hertzsprung intervals in Section~\ref{h-intervals}.
In this decomposition, repeated quotient letters are necessarily
inflated by the one-letter Cayley permutation.
Inflation and the substitution decomposition are defined in
Section~\ref{inflation}.
In Section~\ref{enumeration}, we derive a bivariate generating
function identity for semisimple Cayley permutations and enumerate
them.
In Section~\ref{sec:modules}, we view Cayley permutations as
relational structures and show that indecomposable structures
correspond to Cayley permutations that are simple. We also give
the decomposition with a simple quotient, allowing constant components
at repeated quotient letters. We then enumerate simple Cayley permutations.
In Section~\ref{sec:rgf}, we show how inflation works for restricted
growth functions and use this to enumerate the semisimple and simple ones.

The companion paper~\cite{cerbai-claesson-simple-ii} is devoted to hereditary
classes of Cayley permutations. There, we establish that a class of Cayley
permutations with finitely many simple members has a finite basis and an
algebraic generating function.

\section{Hertzsprung intervals and semisimplicity}\label{h-intervals}

Throughout this paper, we will use the following terminology.
Let $\omega$ be a word. We say that $\beta$ is a \emph{factor} of $\omega$ if there
are words $\alpha$ and $\gamma$ such that $\omega=\alpha\beta\gamma$.
If $\alpha$ is empty we also say that $\beta$ is a \emph{prefix}
of $\omega$, and if, in addition, $\gamma$ is nonempty then $\beta$ is a
\emph{proper prefix} of $\omega$. Similarly, if $\gamma$ is empty we say
that $\beta$ is a \emph{suffix} of $\omega$, and that $\beta$ is a
\emph{proper suffix} if, in addition, $\alpha$ is nonempty.
The \emph{standardization} of $\omega$ is the word $\st(\omega)$
obtained by replacing the smallest letter of $\omega$ by $1$, the
next smallest by $2$, and so on, keeping repetitions equal.
Note that $\st(\omega)$ is a Cayley permutation.

The second author~\cite{ClaHertz22} calls a permutation
$\tau$ of length $k$ a \emph{Hertzsprung factor} of a permutation
$\pi$ if there is a nonnegative integer $c$ and a factor $\beta=b_1b_2\cdots b_k$
of $\pi$ such that $b_i-\tau(i) = c$ for each $i\in[k]$. That is,
the factor $\beta$ is a translate $\tau+c$ of $\tau$.

We shall extend this definition to Cayley permutations.
Let $\omega=w_1\cdots w_n\in\Cay_n$. If $I=[a,b]\subseteq[n]$ is a
nonempty positional interval, write
\[
 \omega|_I=w_aw_{a+1}\cdots w_b
 \qquad\text{and}\qquad
 \omega(I)=\{w_a,w_{a+1},\ldots,w_b\}
\]
for the factor of $\omega$ at positions $I$ and for its set of values.
More generally, if $S=\{i_1<\cdots<i_r\}\subseteq[n]$, write
$\omega|_S=w_{i_1}\cdots w_{i_r}$ for the corresponding subword
and $\omega(S)$ for its values.
We say that $I$ is a \emph{Hertzsprung interval}, or \emph{H-interval},
of $\omega$ if there is a value interval $J=[c,d]$ such that
\begin{equation}\label{eq:h-interval}
 \omega(I)=J
 \quad\text{and}\quad
 \omega^{-1}(J)=I.
\end{equation}
Following Brignall's~\cite{Bri10} convention for permutation intervals,
we also regard the empty set as an H-interval of every Cayley permutation.
Thus, a nonempty H-interval occupies consecutive positions and uses every value
in a consecutive range, which we sometimes call a \emph{value band}.
Additionally, it contains every copy in $\omega$ of every value in $J$ due
to the second equality, which we call the \emph{saturation condition}.
In particular, $[n]$ is always an H-interval. A singleton
$[i,i]=\{i\}$ is an H-interval precisely when the value $\omega(i)$ occurs
nowhere else in $\omega$.
These two types of singletons (unique/repeated) will play different roles in
inflations (see Section~\ref{inflation}).
Note also that the saturation condition is vacuous for permutations,
so the usual definition of permutation interval is obtained as
a special case of the one just given.
Unlike in a permutation, the same factor may appear more than once
in a Cayley permutation. For example, $12$ appears in $1212$ at both
$[1,2]$ and $[3,4]$. We therefore specify an occurrence by its
positional interval.

Let $\tau\in\Cay$. An \emph{occurrence} of $\tau$ in $\omega$ is an
H-interval $I$ of $\omega$ where $\st(\omega|_I)=\tau$. Equivalently, for nonempty $\tau$, we have
\[
 \omega|_I=\tau+c
\]
for an integer $c\geq0$ and no entry of $\omega$ outside $I$ takes a value
in $[c+1,c+\max(\tau)]$. We call $\tau$ a \emph{Hertzsprung factor} of
$\omega$. Every H-interval is an occurrence
of exactly one pattern, namely $\st(\omega|_I)$, and a singleton H-interval
is an occurrence of $1$.
The empty pattern has the single occurrence $\emptyset$ in every word.

\begin{example}\label{example-weak-fails}
In $\omega=121$ the factor $\omega|_{[1,2]}=12$ uses consecutive positions and
consecutive values, but the value $1$ occurs again in the third
position. Hence the saturation condition fails, and $[1,2]$ is not
an occurrence of $12$. Similarly, $[1,2]$ is not an
occurrence of $11$ in $111$.
The nonempty H-intervals of $121$ are the singleton $\{2\}$ and the whole
word $[1,3]$, while the only nonempty H-interval of $111$ is $[1,3]$.
\end{example}

The reader familiar with mesh patterns~\cite{BraCla11} and
Cayley-mesh patterns~\cite{CeSortCay} may notice that an
occurrence of $\tau$ is an occurrence of the Cayley-mesh pattern
whose underlying Cayley permutation is $\tau$ and where every box is
shaded except for the four corner boxes. For instance, an occurrence
of $121$ is an occurrence of the Cayley-mesh pattern
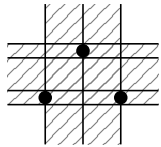
\[
\begin{tikzpicture}[xscale=0.5,yscale=0.62,
  dot/.style={circle,fill,inner sep=0pt,minimum size=5.1pt}]
\fill[NE-lines] (2,0) rectangle (3,3);
\fill[NE-lines] (1,0) rectangle (2,3);
\fill[NE-lines] (0,0.85) rectangle (4,2.15);
\draw [semithick] (0,0.85) -- (4,0.85);
\draw [semithick] (0,1.15) -- (4,1.15);
\draw [semithick] (0,1.85) -- (4,1.85);
\draw [semithick] (0,2.15) -- (4,2.15);
\draw [semithick] (1,0) -- (1,3);
\draw [semithick] (2,0) -- (2,3);
\draw [semithick] (3,0) -- (3,3);
\node[dot] at (1,1) {};
\node[dot] at (2,2) {};
\node[dot] at (3,1) {};
\end{tikzpicture}
\]
and the algebraic condition~\eqref{eq:h-interval} is a
convenient definition of containment in this sense.

Let us briefly motivate the terminology just introduced.
For permutations, a block in the sense of Albert and
Atkinson~\cite{AlbAtk05}, also commonly called an interval~\cite{Bri10},
occupies consecutive positions and
carries a consecutive set of values, and is therefore an H-interval.
We avoid both of the established names. ``Block'' is kept for the blocks
of a ballot
(where a nonempty H-interval is a union of consecutive blocks, see
Proposition~\ref{ballot}). ``Interval'' is used for positional and
value intervals. Interval is also the name under which the
permutation patterns literature treats these sets as the
modules of a relational structure.
This coincidence is special to permutations and does not extend to
Cayley permutations, where the modules form a larger family
than the H-intervals (see Section~\ref{sec:modules}).

\begin{definition}
A Cayley permutation $\omega\in\Cay_n$ is \emph{semisimple} if every
H-interval of $\omega$ is empty, a singleton, or $[n]$;
these H-intervals are called \emph{trivial}. The word is \emph{primitive}~\cite{CCEGcayspec}
if adjacent letters are distinct. It is \emph{simple} if it is semisimple
and primitive, or if $\omega=11$.
\end{definition}

The empty word is semisimple and primitive, hence simple.

As noted before, for permutations
the saturation condition is always satisfied and
both semisimplicity and simplicity restrict to the usual definition of
simplicity. The convention that $11$ is simple makes simplicity equivalent
to relational indecomposability at every length, as shown in
Theorem~\ref{indecomposable}.
All Cayley permutations of length one and two are semisimple by
definition. Example~\ref{example-weak-fails} shows that
$121$ and $111$ are semisimple. An example of a non-semisimple Cayley
permutation is $122$: it contains the proper
H-interval $[2,3]$, which is an occurrence of~$11$ and
occupies value band $[2,2]=\{2\}$.
$121$ and $212$ are simple, whereas $111$ is semisimple but not simple.
There are three semisimple Cayley permutations of length three
and fifteen of length four, listed below:
\begin{align*}
& 111, 212, 121;\\[0.5em]
& 1111, 2212, 1121, 2121, 2122, 3132, 2131, 2132, 3142, 1211, \\
& 1212, 1312, 2312, 2313, 2413.\\
\intertext{Among them, the simple ones are}
& 212, 121;\\[0.5em]
& 2121,3132,2131,2132,3142,1212,1312,2312,2313,2413.
\end{align*}
Recall from Section~\ref{intro} that Cayley permutations encode
ballots. Namely, to the Cayley permutation $\omega\in\Cay_n$
we associate the ballot of $[n]$ defined by
\[
 \varpi=F_1F_2\cdots F_m,
 \qquad F_j=\omega^{-1}(j).
\]
In this context, the fibres, $F_j$, are called \emph{blocks}.
For instance, the Cayley permutation $\omega=211434125$ is associated
with the ballot
\[
\varpi = \{2,3,7\}\;\{1,8\}\;\{5\}\;\{4,6\}\;\{9\}.
\]
We now interpret semisimplicity in terms of ballots, which will make
the proofs of some results in the coming sections
more convenient to state. To start, we characterize their H-intervals.

\begin{proposition}\label{ballot}
  Let $\varpi=F_1F_2\cdots F_m$ be the ballot associated with
  $\omega\in\Cay_n$. Then a nonempty $I\subseteq[n]$ is an H-interval of
  $\omega$ if and only if $I$ is an interval of $[n]$ and
  \[
    I=F_c\cup F_{c+1}\cup\cdots\cup F_d
  \]
  for some $1\leq c\leq d\leq m$.
  Furthermore, such an $I$ is an occurrence of $\tau\in\Cay$ if and
  only if subtracting $\min(I)-1$ from every element of the factor
  $F_c\cdots F_d$ turns it into the ballot of $\tau$.
\end{proposition}

\begin{proof}
  Let $I$ be an H-interval, with value band $J=[c,d]$ as
  in~\eqref{eq:h-interval}. Then the saturation condition gives
  \[
    I=\omega^{-1}(J)=\bigcup_{j=c}^{d}\omega^{-1}(j)=F_c\cup F_{c+1}\cup\cdots\cup F_d,
  \]
  as wanted. Conversely, suppose that $I=F_c\cup\cdots\cup F_d$ is an
  interval of $[n]$, and put $J=[c,d]$. Every block is nonempty,
  so $\omega(I)=J$, and $\omega^{-1}(J)=F_c\cup\cdots\cup F_d=I$.
  Hence~\eqref{eq:h-interval} holds and $I$ is an H-interval.

  For the second claim, let $a=\min(I)$. Position $i$ of $\omega|_I$
  corresponds to position $a+i-1$ of $\omega$. Since $\omega(I)=[c,d]$, the
  letter $j$ occurs in $\st(\omega|_I)$ at position $i$ if and only if
  $\omega(a+i-1)=c+j-1$. Thus, the $j$th block of the ballot of
  $\st(\omega|_I)$ is obtained from $F_{c+j-1}$ by subtracting $a-1$ from
  each element. Therefore, subtracting $a-1$ from every element of
  $F_c\cdots F_d$ gives the ballot of $\st(\omega|_I)$. This is the ballot
  of $\tau$ if and only if $I$ is an occurrence of $\tau$.
\end{proof}

In ballot terms, a nonempty H-interval is a run of consecutive blocks whose
union is an interval of $[n]$. More precisely, a nonempty H-interval $I$
covering the value band $\omega(I)=[c,d]$ corresponds to the blocks
$F_c\cup F_{c+1}\cup\cdots\cup F_d$, whose union is $I$. Because each
block is included in full, saturation is automatic. In particular, a
singleton position is an H-interval precisely when it is a singleton
block.
In the example given before Proposition~\ref{ballot}, the H-intervals
of the Cayley permutation $\omega=211434125$ correspond to those of the
ballot $\varpi = \{2,3,7\}\{1,8\}\{5\}\{4,6\}\{9\}$ as follows:
\begin{align*}
[1,9]&\quad\longleftrightarrow\quad \varpi;\\
[1,8]&\quad\longleftrightarrow\quad \{2,3,7\}\{1,8\}\{5\}\{4,6\};\\
[4,6]&\quad\longleftrightarrow\quad \{5\}\{4,6\};\\
[5,5]&\quad\longleftrightarrow\quad \{5\};\\
[9,9]&\quad\longleftrightarrow\quad \{9\}.
\end{align*}

Thus, semisimplicity for ballots can be stated as follows.

\begin{corollary}\label{ballot-simple}
  Let $\omega\in\Cay_n$ and let $\varpi$ be the corresponding ballot.
  Then $\omega$ is semisimple if and only if the only unions of consecutive
  blocks of $\varpi$ that are intervals of~$[n]$ are the empty set, the singleton
  blocks and $[n]$ itself.
\end{corollary}

For instance, $121$ has ballot $\{1,3\}\{2\}$.  The union $\{1,3\}$
is not an interval of $[3]$, so the only nonempty H-intervals are the
singleton block $\{2\}$ and $[3]$, and $121$ is semisimple.
More generally, in a semisimple Cayley permutation, no block of size two
or more may be an interval of $[n]$ unless it is all of $[n]$.

\section{Inflation and the substitution decomposition}\label{inflation}

Given a permutation $\sigma$ of length $k$ and nonempty
permutations $\alpha_1,\dots,\alpha_k$, the inflation of $\sigma$
by $\alpha_1,\dots,\alpha_k$ is the permutation
$\sigma[\alpha_1,\dots,\alpha_k]$ obtained by replacing each entry
$\sigma(i)$ of $\sigma$ with an interval that is order isomorphic
to $\alpha_i$, keeping the relative order of the $\sigma(i)$'s.

We shall define the inflation of Cayley permutations similarly,
but with one additional constraint:
only letters appearing exactly once in the Cayley permutation
may be inflated by nonconstant words. This prevents
positions carrying the same value from being expanded into
competing value bands, which would make the process ambiguous.
Formally, given $\sigma\in\Cay_k$ 
we associate a nonempty Cayley permutation
$\alpha_i$ with each position $i$, subject to the restriction
\begin{equation}\label{eq:constant-components}
  |\sigma^{-1}(\sigma(i))|>1
  \quad\Longrightarrow\quad \alpha_i=1^{r_i}\text{ for some }r_i\geq1.
\end{equation}
The \emph{inflation} $\sigma[\alpha_1,\ldots,\alpha_k]$ is the
Cayley permutation obtained by replacing each entry $\sigma(i)$ with
a factor order isomorphic to $\alpha_i$, while preserving the relative order of
the entries of $\sigma$. A singleton quotient fibre is replaced by a consecutive
value band; a repeated fibre is replaced by a single value.
We call the positional interval occupied by the entries
corresponding to $\alpha_i$ the $i$th \emph{part} of the inflation.

For the decomposition into H-intervals we impose the stronger restriction
\begin{equation}\label{eq:repeat-restriction}
  |\sigma^{-1}(\sigma(i))|>1 \quad\Longrightarrow\quad \alpha_i=1.
\end{equation}
Under this restriction every nonsingleton part is an H-interval.
Throughout the remainder of this section and Section~\ref{enumeration},
inflations satisfy~\eqref{eq:repeat-restriction}.
For example,
\[
  312143[1,1,221,1,12,1] = 4\, 1\, 332\, 1\, 56\, 4.
\]
The two copies of $1$ and the two copies of $3$ are forced to
remain single letters due to~\eqref{eq:repeat-restriction}; the
unique copy of $2$ is inflated by $221$ and becomes the factor $332$;
and the unique copy of $4$ is inflated by $12$ to the factor $56$.
The six parts of the inflation are
\begin{equation}\label{eq:parts}
\{1\},\{2\}, [3,5], \{6\}, [7,8], \{9\}.
\end{equation}

Now, let $\omega$ be a Cayley permutation. An \emph{H-interval
decomposition} of $\omega$ is a partition of its positions into consecutive
intervals, each of which is an H-interval or a single position.
Contracting the parts of an H-interval decomposition (and
standardizing) yields a Cayley permutation, called the \emph{quotient}.
Repeated letters in the quotient come only from singleton parts.
This distinction matches the two cases of~\eqref{eq:repeat-restriction}:
an H-interval part contracts to a letter that occurs once in the
quotient, whereas a singleton part whose value occurs elsewhere
contracts to a repeated letter.
Contracting a Cayley permutation to obtain a quotient is the inverse
of inflating the quotient.
For instance, an H-interval decomposition of $\omega=413321564$ is
given by the six parts of~\eqref{eq:parts}. In general, the
H-interval decomposition is not unique. The same $\omega$ admits
the H-interval decomposition
\[
4\, 13321\, 56\, 4
\quad\longrightarrow\quad
\{1\},[2,6],[7,8],\{9\}
\]
and contracting its parts yields the quotient $2132$. Conversely,
we see that
\[
2132[1,13321,12,1] = 4\, 13321\, 56\, 4 = \omega.
\]
Although H-interval decompositions need not be unique, requiring a
semisimple quotient leaves only a controlled ambiguity: the quotient is
unique, and the parts are unique except for the quotients $12$ and
$21$, where an indecomposability condition restores uniqueness. This
will be established below. We first record two elementary closure properties
of H-intervals.

\begin{lemma}\label{union-intersection}
Let $I$ and $I'$ be two H-intervals of a Cayley permutation $\omega$
such that $I\cap I'\neq\emptyset$.
\begin{enumerate}
\item The union $I\cup I'$ is an H-interval.
\item The intersection $I\cap I'$ is an H-interval with
  $\omega(I\cap I')=\omega(I)\cap \omega(I')$.
\end{enumerate}
\end{lemma}
\begin{proof}
  We use the ballot viewpoint of Proposition~\ref{ballot}.
  Each H-interval of $\omega$ is a union
  of consecutive blocks of the ballot $\varpi$ that form a set of
  consecutive elements. More precisely, write $\omega(I)=[c,d]$ and
  $\omega(I')=[c',d']$, so that $I$ and $I'$ are the runs of blocks
  $F_c\cdots F_d$ and $F_{c'}\cdots F_{d'}$.
  Suppose that $I$ and $I'$ overlap and let $\ell\in I\cap I'$.
  Then $F_c\cdots F_d$ and $F_{c'}\cdots F_{d'}$ share the block
  $F_{\omega(\ell)}$. Hence their union is again a union of consecutive
  blocks that form a set of consecutive elements, and
  Proposition~\ref{ballot} settles the first claim.
  Similarly, $I\cap I'$ is the run of blocks indexed
  by $[c,d]\cap[c',d']$, so Proposition~\ref{ballot} applies again
  and gives the second claim.
\end{proof}

Define direct and skew sums by
\[
  \alpha\oplus\beta=12[\alpha,\beta]
  \quad\text{and}\quad
  \alpha\ominus\beta=21[\alpha,\beta].
\]
A Cayley permutation is \emph{plus indecomposable} if it cannot be
expressed as a direct sum of shorter Cayley permutations.
Analogously, it is \emph{minus indecomposable} if it cannot be
expressed as a skew sum of shorter Cayley permutations.

The next two lemmas are used in the proof of
Theorem~\ref{substitution-thm}.

\begin{lemma}\label{extreme-cut}
  Let $I=[1,b]$ be a nonempty proper H-interval of $\omega\in\Cay_n$, and suppose
  that its value band $J$ contains $1$ or $\max(\omega)$. Then $[b+1,n]$
  is an H-interval of $\omega$. Furthermore, writing
  $\alpha=\st(\omega|_{[1,b]})$ and $\beta=\st(\omega|_{[b+1,n]})$, we have
  \[
    \omega=\begin{cases}
      \alpha\oplus\beta & \text{if }1\in J,\\
      \alpha\ominus\beta & \text{if }\max(\omega)\in J.
    \end{cases}
  \]
\end{lemma}
\begin{proof}
  Let $m=\max(\omega)$ and write the ballot of $\omega$ as
  $\varpi=F_1F_2\cdots F_m$. Write $J=[c,d]$. By
  Proposition~\ref{ballot},
  \[
    I=F_c\cup F_{c+1}\cup\cdots\cup F_d.
  \]
  Since the blocks partition $[n]$ and $I=[1,b]$, we have
  \[
    [b+1,n]=
    \begin{cases}
      F_{d+1}\cup\cdots\cup F_m & \text{if }c=1,\\
      F_1\cup\cdots\cup F_{c-1} & \text{if }d=m.
    \end{cases}
  \]
  These are the only cases under the hypothesis, and they are mutually
  exclusive because $I$ is proper. In either case, $[b+1,n]$ is a
  positional interval that is the union of consecutive blocks, so it is
  an H-interval by Proposition~\ref{ballot}.

  If $c=1$, every value in $I$ is smaller than every value in
  $[b+1,n]$, and therefore $\omega=\alpha\oplus\beta$. If $d=m$, every
  value in $I$ is larger than every value in $[b+1,n]$, and therefore
  $\omega=\alpha\ominus\beta$.
\end{proof}

\begin{lemma}\label{project}
  Let $\omega=\sigma[\alpha_1,\ldots,\alpha_k]$ be an inflation with
  parts $C_1,\ldots,C_k$, and let $T\subseteq[k]$ be a nonempty positional
  interval. Let $C=\bigcup_{i\in T}C_i$. Then
  \[
    \text{$C$ is an H-interval of $\omega$} \quad\Longleftrightarrow\quad
    \text{$T$ is an H-interval of $\sigma$.}
  \]
\end{lemma}
\begin{proof}
  For each value $j$ of $\sigma$, let $J_j$ be the value band of $\omega$
  corresponding to $j$. These bands partition $[\max(\omega)]$, with
  $\max J_j<\min J_{j+1}$ for $1\leq j<\max(\sigma)$, and
  \[
    \omega(C_i)=J_{\sigma(i)}.
  \]

  First suppose that $C$ is an H-interval of $\omega$.
  Let $a=\min\sigma(T)$ and $b=\max\sigma(T)$.
  Since $a,b\in\sigma(T)$, we have $J_a\cup J_b\subseteq \omega(C)$.
  As $\omega(C)$ is a value interval, $J_j\subseteq \omega(C)$ for every
  $a\leq j\leq b$. If $\sigma(i)\in[a,b]$, it follows that
  $\omega(C_i)\subseteq \omega(C)$. The saturation of $C$ then gives
  $C_i\subseteq C$, and hence $i\in T$. The reverse implication
  $i\in T\Rightarrow\sigma(i)\in[a,b]$ follows from the definitions
  of $a$ and $b$. Therefore,
  \[
    \sigma^{-1}([a,b])=T.
  \]
  Since $\sigma$ is a Cayley permutation, every value in $[a,b]$
  occurs in $\sigma$ and therefore, by this equality, in $\sigma(T)$.
  Thus $\sigma(T)=[a,b]$, and $T$ is an H-interval of $\sigma$.

  Conversely, suppose that $T$ is an H-interval of $\sigma$, with
  $\sigma(T)=[a,b]$ and $\sigma^{-1}([a,b])=T$. Let
  \[
    J=J_a\cup J_{a+1}\cup\cdots\cup J_b.
  \]
  Because the value bands $J_1,\ldots,J_{\max(\sigma)}$ cover
  $[\max(\omega)]$, the union $J=J_a\cup\cdots\cup J_b$ equals
  $[\min J_a,\max J_b]$ and is therefore a value interval. Moreover,
  \[
    \omega(C)=\bigcup_{i\in T}\omega(C_i)
        =\bigcup_{i\in T}J_{\sigma(i)}=J.
  \]
  Every position $p$ of $\omega$ belongs to exactly one part $C_i$, and
  $\omega(p)\in J$ if and only if $\sigma(i)\in[a,b]$. Since
  $\sigma^{-1}([a,b])=T$, this is equivalent to $i\in T$. Hence
  $\omega^{-1}(J)=C$. Finally, $C$ is a positional interval because it is
  a union of consecutive parts. Thus $C$ is an H-interval of $\omega$.
\end{proof}

\begin{theorem}\label{substitution-thm}
  Let $\omega$ be a Cayley permutation of length at least two.
  Then there is a unique semisimple Cayley permutation $\sigma$ of length
  $k\ge 2$ such that
  \[
    \omega=\sigma[\alpha_1,\ldots,\alpha_k],
  \]
  for some nonempty Cayley permutations $\alpha_1,\ldots,\alpha_k$
  satisfying~\eqref{eq:repeat-restriction}.
  If the quotient $\sigma\notin\{12,21\}$, then the components
  $\alpha_i$ are uniquely determined.
  If $\sigma=12$, they are unique if we require $\alpha_1$ to be
  plus indecomposable. If $\sigma=21$, they are unique if we require
  $\alpha_1$ to be minus indecomposable.
\end{theorem}

\begin{proof}
  Assume $\omega\in\Cay_n$ with $n\geq 2$. We first prove existence. Among all
  inflations $\omega=\sigma[\alpha_1,\ldots,\alpha_k]$ with $k\geq2$, choose one
  with $k$ minimal. Such an inflation exists, since $\omega=\omega[1,\ldots,1]$.
  If $\sigma$ were not semisimple, it would have an H-interval $T$ with
  $2\leq|T|<k$. By Lemma~\ref{project}, the union of the parts indexed by
  $T$ is an H-interval of $\omega$. Merge these parts into one, standardizing
  the factor on their union. The merged value band shares no value with
  any remaining part, so its letter in the new quotient occurs only once;
  the restrictions at all other repeated quotient letters are preserved.
  The resulting inflation has $k-|T|+1$ parts. Since
  $2\leq |T|<k$, this number lies between $2$ and $k-1$, contradicting
  the minimality of $k$. Hence $\sigma$ is semisimple.

  We now prove uniqueness.
  Suppose first that $\omega$ is plus decomposable. Thus, for some
  $b\in[n-1]$, we have $\omega(i)<\omega(j)$ for all positions $i\leq b<j$.
  Consider any expression of $\omega$ as an inflation of a semisimple quotient
  with at least two parts. If $b$ is a boundary between
  parts, the parts on either side express the quotient as a direct sum.
  If $b$ lies inside a part, that part has at least two positions, so its
  quotient letter occurs only once by~\eqref{eq:repeat-restriction}. Its
  value band contains values from positions on both sides of $b$.
  Since the value bands of distinct quotient letters are disjoint, every
  earlier quotient letter is smaller and every later quotient letter is
  larger. Separating the quotient immediately before or after this letter,
  choosing the option that leaves both sides nonempty, expresses it as a
  direct sum. If the quotient had length at least three, one summand would
  be a proper H-interval of length at least two. Thus a semisimple quotient
  must be $12$.

  Choose the smallest $b\in[n-1]$ with this property. The
  standardized prefix is plus indecomposable, or else its direct-sum
  decomposition would give a smaller such value of $b$. For any other
  $b'>b$ with this property, the standardized prefix ending at $b'$
  is plus decomposable, with its first summand ending at $b$. Hence $b$
  gives the unique decomposition whose first component is plus indecomposable.
  Reversing the value comparisons proves that a minus decomposable Cayley
  permutation has semisimple quotient $21$ and a unique decomposition whose
  first component is minus indecomposable.

  It remains to consider $\omega$ that is neither plus nor minus decomposable.
  Its maximal nonempty proper H-intervals are pairwise disjoint. Indeed, if two
  distinct ones overlapped, their union would be $[n]$ by
  Lemma~\ref{union-intersection} and maximality. One would be a proper
  prefix and the other a proper suffix. Their value bands cover
  $[\max(\omega)]$, and neither contains both extreme values, since saturation
  would make the corresponding H-interval equal to $[n]$. The prefix must
  therefore contain $1$ or $\max(\omega)$, so Lemma~\ref{extreme-cut} gives a
  direct-sum or skew-sum decomposition, a contradiction.

  These maximal H-intervals, together with the singleton sets formed by the
  remaining positions, therefore partition $[n]$. Consider any inflation of $\omega$
  with a semisimple quotient of length at least two. Each nonsingleton part is
  a proper H-interval and is therefore contained in a maximal one; each
  singleton part is contained in the unique member of the partition that
  contains its position. Thus the parts refine this partition. If a maximal
  H-interval were the union of at least two parts, Lemma~\ref{project} would
  give a proper H-interval of length at least two in the quotient, contradicting its
  semisimplicity. Consequently, the parts are exactly the maximal H-intervals
  and the uncovered singleton positions. Their standardized factors and
  their quotient are uniquely determined.
\end{proof}

The remaining semisimple quotient of length two, $\sigma=11$, causes no
ambiguity. Its two positions carry the same value, so
\eqref{eq:repeat-restriction} forces both components to be $1$; the
only permitted inflation is the Cayley permutation $11=11[1,1]$ itself.

\section{Enumeration of semisimple Cayley permutations}\label{enumeration}

We shall adapt the enumeration of simple
permutations~\cite{AlbAtkKla03} to Cayley permutations. The main
difference is that, under~\eqref{eq:repeat-restriction}, a position carrying a repeated letter in the quotient
can only be inflated by the one-letter Cayley permutation $1$.
Given $\omega\in\Cay$, let
\[
  \mu(\omega)=\#\{j:|\omega^{-1}(j)|=1\}
\]
denote the number of singleton fibres of $\omega$, and define the
bivariate ordinary generating function
\[
  B(x,t)=\sum_{\omega\in\Cay_+}x^{|\omega|}t^{\mu(\omega)}
  =\sum_{n\geq1}B_n(t)x^n,
  \qquad
  B_n(t)=\sum_{\omega\in\Cay_n}t^{\mu(\omega)}.
\]
The ballot viewpoint is convenient to determine an exponential
generating function for the polynomials $B_n(t)$.

\begin{proposition}\label{B-egf}
  The polynomials $B_n(t)$ are determined by
  \[
    1+\sum_{n\geq1}B_n(t)\frac{x^n}{n!} \,=\, \frac{1}{2-e^x-(t-1)x}.
  \]
\end{proposition}
\begin{proof}
  View a Cayley permutation as a ballot.  A singleton block has
  exponential generating function $tx$, while a block of size at least two has
  exponential generating function $e^x-1-x$.  Taking a sequence of such
  nonempty blocks gives
  \[
    \frac{1}{1-\bigl(tx+e^x-1-x\bigr)}
    =\frac{1}{2-e^x-(t-1)x}.\qedhere
  \]
\end{proof}

The first few of these polynomials are listed in Table~\ref{B-table}.

\begin{table}[ht]
  \centering
  \renewcommand{\arraystretch}{1.2}
  \begin{tabular}{c@{\qquad}l}
    \toprule
    $n$ & $B_n(t)$\\
    \midrule
    1 & $t$\\
    2 & $1+2t^2$\\
    3 & $1+6t+6t^3$\\
    4 & $7+8t+36t^2+24t^4$\\
    5 & $21+100t+60t^2+240t^3+120t^5$\\
    6 & $141+372t+1170t^2+480t^3+1800t^4+720t^6$\\
    7 & $743+3584t+5166t^2+13440t^3+4200t^4+15120t^5+5040t^7$\\
    \bottomrule
  \end{tabular}
  \caption{The polynomials $B_n(t)=\sum_{\omega\in\Cay_n}t^{\mu(\omega)}$ for $n\leq7$.}
  \label{B-table}
\end{table}

Let $\stirling{n}{k}$ denote the Stirling number of the second kind, that
is, the number of partitions of $[n]$ into $k$ blocks. Let
$\stirling{n}{k}_{\!\geq2}$ denote the number of such partitions in which
every block has size at least two.

\begin{proposition}\label{B-explicit}
  For $n\geq1$,
  \[
    B_n(t)
    =\sum_{k=1}^{n}k!\sum_{j=0}^{k}\binom{n}{j}\stirling{n-j}{k-j}(t-1)^j.
  \]
  Moreover, for $0\leq j\leq n$,
  \[
    [t^j]\,B_n(t)
    =\binom{n}{j}\sum_{i\geq0}(i+j)!\,\stirling{n-j}{i}_{\!\geq2}.
  \]
\end{proposition}

\begin{proof}
  Using
  \[
    (e^x-1)^r=r!\sum_{m\geq0}\stirling{m}{r}\frac{x^m}{m!},
  \]
  we can expand the generating function in Proposition~\ref{B-egf} as
  \begin{align*}
    1+\sum_{n\geq1}B_n(t)\frac{x^n}{n!}
    &=\sum_{k\geq0}\bigl(e^x-1+(t-1)x\bigr)^k\\
    &=\sum_{k\geq0}\sum_{j=0}^{k}\binom{k}{j}
      (t-1)^jx^j(e^x-1)^{k-j}\\
    &=\sum_{k\geq0}\sum_{j=0}^{k}\binom{k}{j}(t-1)^jx^j
      (k-j)!\sum_{m\geq0}\stirling{m}{k-j}\frac{x^m}{m!}\\
    &=1+\sum_{n\geq1}\left(
      \sum_{k=1}^{n}k!\sum_{j=0}^{k}\binom{n}{j}
      \stirling{n-j}{k-j}(t-1)^j\right)\frac{x^n}{n!}.
  \end{align*}
  Comparing coefficients of $x^n/n!$ proves the first formula.
  For the second formula, use
  \[
    (e^x-1-x)^i
    =i!\sum_{m\geq0}\stirling{m}{i}_{\!\geq2}\frac{x^m}{m!}.
  \]
  Extracting the coefficient of $t^j$ from the generating function and
  expanding gives
  \begin{align*}
    [t^j]\left(1+\sum_{n\geq1}B_n(t)\frac{x^n}{n!}\right)
    &=[t^j]\sum_{r\geq0}\bigl(tx+e^x-1-x\bigr)^r\\
    &=x^j\sum_{i\geq0}\binom{i+j}{j}(e^x-1-x)^i\\
    &=x^j\sum_{i\geq0}\binom{i+j}{j}i!
      \sum_{m\geq0}\stirling{m}{i}_{\!\geq2}\frac{x^m}{m!}\\
    &=\sum_{n\geq j}\left(\binom{n}{j}\sum_{i\geq0}(i+j)!
      \stirling{n-j}{i}_{\!\geq2}\right)\frac{x^n}{n!}.
  \end{align*}
  Comparing coefficients of $x^n/n!$ proves the second formula.

  Alternatively, the latter formula follows by a direct combinatorial argument.
  A ballot of $[n]$ with exactly $j$ singletons is obtained by choosing the $j$
  elements in those blocks, partitioning the remaining $n-j$ elements into
  $i$ blocks of size at least two, and linearly ordering all $i+j$ blocks.
  For each $i$, these choices give
  $\binom{n}{j}(i+j)!\stirling{n-j}{i}_{\!\geq2}$ ballots, and summing over
  $i$ gives the formula.
\end{proof}

To enumerate semisimple Cayley permutations and several of their subclasses at
once, let $Q$ be a property of nonempty Cayley permutations such that $Q(1)$
holds and, for every inflation satisfying~\eqref{eq:repeat-restriction},
\begin{equation}\label{eq:inflation-compatible}
  Q\bigl(\sigma[\alpha_1,\ldots,\alpha_k]\bigr)
  \quad\Longleftrightarrow\quad
  Q(\sigma)\ \text{and}\ Q(\alpha_i)\text{ for every }i\in [k].
\end{equation}
Define
\[
  F_Q(x,t)=\sum_{\substack{\omega\in\Cay_+\\Q(\omega)}}x^{|\omega|}t^{\mu(\omega)},
  \qquad
  H_Q(x,y)=\sum_{\substack{\sigma\text{ semisimple},\ |\sigma|\geq3\\Q(\sigma)}}
    x^{|\sigma|-\mu(\sigma)}y^{\mu(\sigma)}.
\]
The two variables in $H_Q$ distinguish positions in non-singleton fibres,
which can only be inflated trivially, from positions in singleton fibres,
which may be inflated by arbitrary components.

\begin{lemma}\label{enumeration-compatible}
  For a property $Q$ satisfying $Q(1)$ and~\eqref{eq:inflation-compatible},
  let $\chi_{\omega}$ be $1$ if $Q(\omega)$ holds and $0$ otherwise, for each word $\omega$.
  Then
  \[
    F_Q(x,t)=tx+\chi_{11}x^2+
      \frac{\bigl(\chi_{12}+\chi_{21}\bigr)F_Q(x,t)^2}
        {1+F_Q(x,t)}+H_Q\bigl(x,F_Q(x,t)\bigr).
  \]
\end{lemma}

\begin{proof}
  Write $F=F_Q(x,t)$. We partition the words counted by $F$ according to the
  unique semisimple quotient supplied by Theorem~\ref{substitution-thm}.

  The word $1$ has weight $tx$. For the quotient $11$, the restriction
  in~\eqref{eq:repeat-restriction} forces both components to be $1$, so this
  quotient contributes $\chi_{11}x^2$.

  Now suppose that $Q(12)$ holds, and let $I_+$ be the generating function
  for nonempty plus indecomposable words satisfying $Q$. Every word counted
  by $F$ is uniquely either plus indecomposable or a direct sum
  $12[\alpha,\beta]$ in which $\alpha$ is plus indecomposable. Moreover,
  by~\eqref{eq:inflation-compatible}, both components satisfy $Q$. Since the
  weight is multiplicative under direct sums, this gives
  \[
    F=I_++I_+F=I_+(1+F)
    \quad\Longrightarrow\quad
    I_+ = \frac{F}{1+F}.
  \]
  Hence the words whose quotient is $12$ contribute
  \[
    I_+F=\frac{F^2}{1+F}.
  \]
  If $Q(12)$ does not hold,~\eqref{eq:inflation-compatible} excludes every
  nontrivial direct sum. The same argument with skew sums and minus
  indecomposable words applies to the quotient $21$. Thus the quotients $12$
  and $21$ together contribute
  \[
    \bigl(\chi_{12}+\chi_{21}\bigr)\frac{F^2}{1+F}.
  \]

  It remains to consider a semisimple quotient $\sigma$ of length at least three.
  If a value occurs more than once in $\sigma$, then
  \eqref{eq:repeat-restriction} forces the component at each of its positions
  to be $1$; each such position contributes $x$. A value that occurs once may
  instead be inflated by any word counted by $F$, and therefore contributes
  $F$. There are $|\sigma|-\mu(\sigma)$ positions of the first kind and
  $\mu(\sigma)$ of the second. Hence the inflations of $\sigma$ contribute
  \[
    x^{|\sigma|-\mu(\sigma)}F^{\mu(\sigma)}.
  \]
  By~\eqref{eq:inflation-compatible}, such inflations satisfy $Q$ exactly when
  $\sigma$ does. Summing over the eligible quotients gives $H_Q(x,F)$.
  Since $F$ has zero constant term, the displayed contribution has degree at
  least $|\sigma|$, hence this substitution is well defined as a formal series.
  Adding the four contributions proves the formula.
\end{proof}

We now take $Q$ to be the property that holds for every nonempty Cayley
permutation. In this unrestricted case, $F_Q=B$ and
$\chi_{11}=\chi_{12}=\chi_{21}=1$. Write $\widetilde S_{\geq3}=H_Q$; explicitly,
\[
  \widetilde S_{\geq3}(x,y)
  =\sum_{\substack{\sigma\text{ semisimple}\\|\sigma|\geq3}}
    x^{|\sigma|-\mu(\sigma)}y^{\mu(\sigma)}.
\]
Lemma~\ref{enumeration-compatible} now specializes as follows.

\begin{theorem}\label{bivariate}
  The generating function for semisimple Cayley permutations satisfies
  \[
    B(x,t)=tx+x^2+\frac{2B(x,t)^2}{1+B(x,t)}+\widetilde S_{\geq3}\bigl(x,B(x,t)\bigr).
  \]
\end{theorem}

The identity in Theorem~\ref{bivariate} rewrites as
\[
    \widetilde S_{\geq3}\bigl(x,B(x,t)\bigr)
    =B(x,t)-tx-x^2-\frac{2B(x,t)^2}{1+B(x,t)}.
\]

In their enumeration of simple permutations, Albert, Atkinson, and
Klazar~\cite{AlbAtkKla03} use the coefficients of the compositional inverse of
$\sum_{n\ge 1}n!x^n$. These coefficients were first studied by
Comtet~\cite[p.~171]{Com74}.
We use the following general inversion lemma.

\begin{lemma}\label{U-lemma}
  Let $F(x,t)=tx+\sum_{k\geq2}F_k(t)x^k$, where $F_k(t)\in\mathbb{Z}[t]$.
  There is a unique formal power series $V(x)\in\mathbb{Z}[[x]]$ such that
  $F\bigl(x,V(x)\bigr)=x$. Its coefficients $v_n=[x^n]V(x)$ are given by
  $v_0=1$ and, for $n\geq1$, by
  \[
    v_n=-\sum_{k=2}^{n+1}[x^{n+1-k}]F_k\bigl(V(x)\bigr),
  \]
  where the right-hand side depends only on $v_0,\ldots,v_{n-1}$.
\end{lemma}

\begin{proof}
  Substituting $t=V(x)$ into the given expansion of $F$ turns the required
  identity $F(x,V(x))=x$ into
  \[
    xV(x)+\sum_{k\geq2}F_k\bigl(V(x)\bigr)x^k=x.
  \]
  Dividing by $x$ and rearranging gives
  \[
    V(x)=1-\sum_{k\geq2}F_k\bigl(V(x)\bigr)x^{k-1}.
  \]
  Its constant term gives $v_0=1$. For $n\geq1$, extracting the coefficient
  of $x^n$ gives the stated recurrence; terms with $k>n+1$ make no
  contribution. For each $2\leq k\leq n+1$, the coefficient
  $[x^{n+1-k}]F_k(V(x))$ depends only on
  $v_0,\ldots,v_{n+1-k}$, and $n+1-k\leq n-1$. The recurrence therefore
  determines $v_n$ from the preceding coefficients. Starting with $v_0=1$
  constructs a unique solution coefficient by coefficient, and all its
  coefficients are integers because every $F_k$ has integer coefficients.
\end{proof}

\begin{corollary}\label{univariate-compatible}
  Under the hypotheses of Lemma~\ref{enumeration-compatible}, let $\widetilde S_Q(x)$
  count the nonempty semisimple Cayley permutations satisfying $Q$, and let
  $V_Q(x)$ be the unique series with $F_Q(x,V_Q(x))=x$. Then
  \[
    \widetilde S_Q(x)=2x-xV_Q(x)
      +\frac{\bigl(\chi_{12}+\chi_{21}\bigr)x^3}{1+x}.
  \]
\end{corollary}

\begin{proof}
  Since $Q(1)$ holds, the coefficient of $x$ in $F_Q(x,t)$ is $t$, so
  Lemma~\ref{U-lemma} supplies $V_Q$. Substituting $t=V_Q(x)$ in
  Lemma~\ref{enumeration-compatible} gives
  \[
    H_Q(x,x)=x-xV_Q(x)-\chi_{11}x^2
      -\frac{\bigl(\chi_{12}+\chi_{21}\bigr)x^2}{1+x}.
  \]
  The series $H_Q$ includes only semisimple words of length at least three. Every
  Cayley permutation of length one or two is semisimple, so, using $Q(1)$, we have
  \[
    H_Q(x,x)=\widetilde S_Q(x)-x
      -\bigl(\chi_{11}+\chi_{12}+\chi_{21}\bigr)x^2.
  \]
  Combining the two identities proves the formula; the contributions from
  $11$ cancel.
\end{proof}

In the unrestricted case, where $Q$ holds for every nonempty Cayley
permutation, write $U=V_Q$. Thus $U$ is the unique series with
\[
  B\bigl(x,U(x)\bigr)=x.
\]
Write $u_n=[x^n]U(x)$.
By the recurrence in Lemma~\ref{U-lemma}, for example,
$u_1=-B_2(u_0)=-B_2(1)=-3$.
The first terms of $U(x)$ are
\[
  \begin{split}
    U(x)={}&1-3x-x^2-17x^3-95x^4-871x^5-9189x^6\\
    &{}-112111x^7-1537157x^8-23315921x^9+ \cdots
  \end{split}
\]

\begin{corollary}\label{univariate}
  Let $\tilde s_n$ be the number of semisimple Cayley permutations of length $n$
  and let $\widetilde S(x)=\sum_{n\geq1}\tilde s_nx^n$. Then
  \[
    \widetilde S(x)
    =2x+\frac{2x^3}{1+x}-xU(x).
  \]
\end{corollary}

\begin{proof}
  This is Corollary~\ref{univariate-compatible} with $V_Q=U$, since both
  $12$ and $21$ satisfy the unrestricted property.
\end{proof}

Comparing coefficients in Corollary~\ref{univariate} gives
\[
\tilde s_n=-u_{n-1}+2(-1)^{n-1}
\]
for $n\geq3$. Hence
\begin{equation}\label{eq:initial-counts}
  (\tilde s_n)_{n\geq0}
  =1,1,3,3,15,97,869,9191,112109,1537159,23315919,\ldots
\end{equation}
For reference, Table~\ref{counts-table} shows these values,
alongside the Fubini numbers $|\Cay_n|$ and the number of simple Cayley permutations $s_n$ of
Theorem~\ref{primitive-bivariate} below.

\begin{table}[ht]
  \centering
  \renewcommand{\arraystretch}{1.2}
  \begin{tabular}{c@{\qquad}r@{\qquad}r@{\qquad}r}
    \toprule
    $n$ & $|\Cay_n|$ & $\tilde s_n$ & $s_n$\\
    \midrule
    0 & 1 & 1 & 1\\
    1 & 1 & 1 & 1\\
    2 & 3 & 3 & 3\\
    3 & 13 & 3 & 2\\
    4 & 75 & 15 & 10\\
    5 & 541 & 97 & 70\\
    6 & 4683 & 869 & 634\\
    7 & 47293 & 9191 & 6742\\
    8 & 545835 & 112109 & 82306\\
    9 & 7087261 & 1537159 & 1126846\\
    10 & 102247563 & 23315919 & 17050626\\
    \bottomrule
  \end{tabular}
  \caption{Cayley permutations $|\Cay_n|$, semisimple Cayley permutations $\tilde s_n$,
    and simple Cayley permutations $s_n$.}
  \label{counts-table}
\end{table}

\section{Modules and simplicity}\label{sec:modules}

A \emph{relational structure} is a set together with a collection
of binary relations on it. Analogues of simplicity and the substitution
decomposition may be defined for any relational structure
(e.g., graphs and posets).
Schmerl and Trotter~\cite{SchTro93} defined \emph{indecomposable}
relational structures and proved that a relational structure is
indecomposable if and only if it contains no nontrivial
\emph{module} (sometimes called an interval in the
literature). Here, a module is a subset~$M$ of the structure
such that, for every relation $R$, every $i,j\in M$, and every $x\notin M$,
\[
  iRx\ \Longleftrightarrow\ jRx
  \qquad\text{and}\qquad
  xRi\ \Longleftrightarrow\ xRj.
\]
In other words, every two elements of $M$ are ordered with respect
to elements not in $M$ in exactly the same way with respect to all
relations.
A module is nontrivial if it has at least two elements and is a proper subset
of the ground set.
We also allow the empty structure, which is indecomposable since it has
no nontrivial module.
Albert and Atkinson~\cite{AlbAtk05} encoded a permutation $\pi$
of $[n]$ as the poset on $[n]$ where $x\prec y$ if and only if
$x<y$ and $\pi(x)<\pi(y)$, so that $\pi$ is simple exactly when
this poset is indecomposable in the relational structure sense.

The presence of repeated entries requires a slightly different approach. To each
Cayley permutation, we will associate a relational structure so that
indecomposable structures correspond to Cayley permutations that
are simple.
Given a Cayley permutation $\omega\in\Cay_n$, define the
relational structure
\[
  R(\omega)=\bigl([n],\,<_1,\,<_2\bigr)
\]
on the set of positions $[n]$, carrying two binary relations:
the position order $i<_1 j$ if $i<j$, and the value order
$i<_2 j$ if $\omega(i)<\omega(j)$.
Both relations are irreflexive; $<_1$ is a linear order, and $<_2$
is a strict weak order whose ties are the positions of equal value.
The main difference compared to Albert and Atkinson~\cite{AlbAtk05}
is that the two orders are kept apart, and not merged into a
single relation.

The empty set is a module, since its defining condition is vacuous.
The next lemma characterizes the modules of $R(\omega)$.

\begin{lemma}\label{modules}
  The modules of $R(\omega)$ are the H-intervals of $\omega$ together with its
  nonempty positional intervals with constant value, which in the ballot are
  the nonempty positional intervals contained in a single block.
\end{lemma}

\begin{proof}
  The empty case is immediate. Let $M$ be a nonempty module.
  If a position $x\notin M$ lay between two
  positions of $M$, then $y<_1x$ for some $y\in M$ and $x<_1y'$ for some
  $y'\in M$, contradicting the module condition. Thus $M$ is a positional
  interval. If $\omega$ is constant on $M$, then $M$ is contained in a single
  block. Otherwise, the module condition for
  $<_2$, in both directions, forces each position outside $M$ to lie in a block
  below every block intersecting $M$ or above every such block.  In particular,
  no position outside $M$ lies in a block intersecting $M$, so $M$ is a union
  of whole blocks.  Those blocks are consecutive, since a block between two of
  them would be nonempty and its positions neither below nor above all of $M$.
  By Proposition~\ref{ballot}, $M$ is an H-interval.
  Conversely, an H-interval is a module because every block outside it lies
  wholly below or wholly above it. A positional interval contained in a single
  block is also a module.
\end{proof}

Under the encoding employed by Albert and Atkinson~\cite{AlbAtk05},
where $<_1$ and $<_2$ are merged into one relation, the
correspondence between the intervals of a permutation and the
modules of the corresponding relational structure is only partial.
For instance, the positions $2$ and $4$ in $1432$ form a module
of the poset, but not an interval. Keeping $<_1$ and $<_2$
separate restores the exact correspondence in our setting.
Indeed, if $\omega$ is a permutation then the only positional
intervals with constant value are the singletons, which are
H-intervals. Thus, for permutations the modules of $R(\omega)$ are
exactly the H-intervals of $\omega$.
In general, the correspondence between modules and
H-intervals breaks on Cayley permutations:
a repeated value makes each single position carrying that value
a module that is not an H-interval.

\begin{theorem}\label{indecomposable}
  A Cayley permutation $\omega$ is simple if and only if the relational structure
  $R(\omega)$ is indecomposable.
\end{theorem}

\begin{proof}
  For $n\leq2$, every Cayley permutation is simple and there is no
  nontrivial module. Suppose henceforth that $n\geq3$.
  Suppose $\omega$ is simple, and let $M$ be a module with
  $2\leq|M|<n$.  By Lemma~\ref{modules}, $M$ is either an H-interval or a
  constant positional interval. The first is impossible because every
  nonempty H-interval of a semisimple Cayley permutation is a singleton or $[n]$. The second
  is impossible because a constant interval of length at least two contains a
  flat step, which primitivity forbids. No such $M$ exists, so $R(\omega)$ is
  indecomposable.

  Conversely, let $R(\omega)$ be indecomposable.  Every H-interval is a
  module, so every H-interval is empty, a singleton, or $[n]$, which is to
  say that $\omega$ is semisimple.  If $\omega(i)=\omega(i+1)$ for some $i$, then $\{i,i+1\}$ is a
  constant positional interval and hence a module; it is not a singleton, and
  it is not all of $[n]$ because $n\geq3$, contradicting indecomposability.
  Hence $\omega$ is primitive.
\end{proof}

Next we allow the full range of components in the definition of inflation,
subject only to~\eqref{eq:constant-components}. Every part is then a
module by Lemma~\ref{modules}. Conversely, contracting the parts of a
partition into nonempty modules gives such an inflation. For example,
\[
  1121=121[11,1,1],
  \qquad 111=11[1,11].
\]

We first relate the two kinds of quotient.

\begin{lemma}\label{semisimple-runs}
  Every nonconstant semisimple Cayley permutation is obtained uniquely
  from a primitive simple Cayley permutation by replacing positions in
  repeated fibres by nonempty constant runs and leaving positions in
  singleton fibres unchanged.
\end{lemma}

\begin{proof}
  Contract every maximal constant run of a nonconstant semisimple word
  $\omega$. Each contraction preserves semisimplicity since the inverse image
  of an H-interval under the order-preserving contraction map is an
  H-interval with the same value band. Thus a proper H-interval of length
  at least two after contraction would give one before contraction.
  The resulting word $\sigma$ is primitive and semisimple, hence simple.
  A run of length greater than one cannot contain the entire fibre of
  its value, since it would be a proper H-interval of length at least two in
  $\omega$. Its value therefore appears in at least two runs and remains
  repeated in $\sigma$.

  Conversely, expand only positions in repeated fibres of a primitive
  simple word $\sigma$. Saturation forces any nonempty H-interval of the resulting
  word to contain every maximal run it meets. Contracting those runs
  gives an H-interval of $\sigma$. If this is a singleton, its value is
  unique in $\sigma$ and its position was not expanded, so the original
  H-interval is a singleton. Otherwise it is the whole word. Thus the
  expanded word is semisimple. Maximal runs determine the contraction
  and their lengths uniquely.
\end{proof}

\begin{theorem}\label{module-substitution-thm}
  Every Cayley permutation $\omega$ of length at least two has a unique
  simple quotient $\sigma$ of length $k\geq2$ such that
  \[
    \omega=\sigma[\alpha_1,\ldots,\alpha_k]
  \]
  for nonempty components satisfying~\eqref{eq:constant-components}.
  If $\sigma\notin\{11,12,21\}$, the components are unique. For $\sigma=11$
  they are unique on requiring $\alpha_1=1$; for $\sigma=12$ or $21$ they
  are unique on requiring $\alpha_1$ to be plus or minus indecomposable,
  respectively.
\end{theorem}

\begin{proof}
  A constant word has quotient $11$, with the stated normalization.
  No nonconstant quotient can produce a constant word, and an
  inflation of $11$ is constant, so this case is separate.

  For a nonconstant word, take its H-interval decomposition from
  Theorem~\ref{substitution-thm}, with semisimple quotient $\tau$.
  The quotient is nonconstant. Apply Lemma~\ref{semisimple-runs} to
  $\tau$, merging the one-letter components along its constant runs.
  This gives an inflation of its simple contraction $\sigma$.

  Conversely, start with an inflation of a simple nonconstant
  quotient $\sigma$. Replace each position in a repeated fibre by as
  many copies as there are letters in its constant component, while
  leaving each singleton-fibre position unchanged. The resulting
  quotient $\tau$ is semisimple by Lemma~\ref{semisimple-runs}, and the
  same word is its inflation satisfying~\eqref{eq:repeat-restriction}.
  The uniqueness in
  Theorem~\ref{substitution-thm} determines $\tau$ and its components,
  subject to the two sum normalizations. Contracting its maximal runs
  recovers $\sigma$ and the constant components uniquely. The quotients
  $12$ and $21$ have no repeated fibres, so their normalizations agree
  with those in Theorem~\ref{substitution-thm}.
\end{proof}

\section{Enumeration of simple Cayley permutations}

We now enumerate the simple Cayley permutations, which
by Theorem~\ref{indecomposable} are
exactly the Cayley permutations whose associated structure is
indecomposable.
Under~\eqref{eq:repeat-restriction}, primitivity satisfies
\eqref{eq:inflation-compatible}. Indeed, every component
of a primitive inflation is primitive because it is a standardized factor.
Its quotient is primitive too; otherwise, adjacent equal letters in the
quotient (which have forced one-letter components) would give adjacent
equal letters in the inflation.
Conversely, if the quotient and components are primitive, there are no equal
adjacent letters within a component, and adjacent components use disjoint
value bands because their quotient letters differ.

Write
\[
  P(x,t)=\sum_{\substack{\omega\in\Cay_+\\\omega\text{ primitive}}}x^{|\omega|}t^{\mu(\omega)}
  =\sum_{n\geq1}P_n(t)x^n
\]
for the primitive analogue of $B(x,t)$ and let
\[
  S_{\geq3}(x,y)
  =\sum_{\substack{\sigma\text{ simple}\\|\sigma|\geq3}}
   x^{|\sigma|-\mu(\sigma)}y^{\mu(\sigma)}.
\]

\begin{theorem}\label{primitive-bivariate}
  The generating function for simple Cayley permutations satisfies
  \[
    S_{\geq3}\bigl(x,P(x,t)\bigr)
    =P(x,t)-tx-\frac{2P(x,t)^2}{1+P(x,t)}.
  \]
  Let $W(x)$ be the unique formal power series with $P\bigl(x,W(x)\bigr)=x$.
  Then the ordinary generating function
  $S(x)=\sum_{n\geq1}s_n x^n$ for simple Cayley
  permutations is
  \[
  S(x)=2x+x^2+\frac{2x^3}{1+x}-x\mskip1mu W(x).
  \]
\end{theorem}

\begin{proof}
  Apply the results of Section~\ref{enumeration} to the property of being
  primitive, which satisfies~\eqref{eq:inflation-compatible} by the argument
  above. Here $F_Q=P$, the word $11$ is excluded, and both $12$ and $21$ are
  admitted. Lemma~\ref{enumeration-compatible} gives the first identity,
  Lemma~\ref{U-lemma} supplies $W$, and
  Corollary~\ref{univariate-compatible}, with $V_Q=W$, counts the primitive
  semisimple words. Adding $x^2$ for the simple word $11$ gives the second
  identity.
\end{proof}

The additional term $x^2$ in $S$ accounts for $11$, the only
simple Cayley permutation that is not primitive.

To compute the coefficients of $S$, we express $P$ in terms of $B$.
Define
\[
  \widehat{B}(x,t)=\sum_{\omega\in\Cay_+}x^{|\omega|-\mu(\omega)}t^{\mu(\omega)}
  \quad\text{and}\quad
  \widehat{P}(x,t)=\sum_{\substack{\omega\in\Cay_+\\\omega\text{ primitive}}}x^{|\omega|-\mu(\omega)}t^{\mu(\omega)}.
\]
Here $t$ marks positions in singleton fibres and $x$ marks all other positions.
Contracting each maximal constant run to a single letter gives a unique
primitive Cayley permutation. Conversely, we recover the original word by
expanding each position into its corresponding nonempty constant run.
A position in a singleton fibre may expand into a run of length one,
contributing $t$, or a run of length at least two, whose value is then
repeated. Summing the weights of the latter runs gives $x^2/(1-x)$, so the
possible expansions of that position have total weight
\[
  t+\frac{x^2}{1-x},
\]
whereas the possible expansions of each position carrying a repeated value
have total weight $x/(1-x)$. Consequently,
\[
  \widehat{B}(x,t)
  =\widehat{P}\!\left(\frac{x}{1-x},\,t+\frac{x^2}{1-x}\right).
\]
Since
$B(x,t)=\widehat{B}(x,tx)$ and $P(x,t)=\widehat{P}(x,tx)$, we have
\[
  B(x,t)=P\!\left(\frac{x}{1-x},\,t(1-x)+x\right),
\]
and inverting this substitution gives
\[
  P(x,t)=B\!\left(\frac{x}{1+x},\,(1+x)t-x\right).
\]
Proposition~\ref{B-explicit} determines the coefficient polynomials of $B$,
so this identity determines those of $P$. Lemma~\ref{U-lemma} then computes
the coefficients of $W$, and Theorem~\ref{primitive-bivariate} gives those
of $S$.

\begin{corollary}\label{semisimple-simple-series}
  The semisimple and simple quotient series satisfy
  \[
    \widetilde S_{\geq3}(x,y)
      =S_{\geq3}\!\left(\frac{x}{1-x},y\right)
        +\frac{x^3}{1-x}.
  \]
  Consequently,
  \[
    B(x,t)=tx+\frac{x^2}{1-x}
      +\frac{2B(x,t)^2}{1+B(x,t)}
      +S_{\geq3}\!\left(\frac{x}{1-x},B(x,t)\right).
  \]
\end{corollary}

\begin{proof}
  In Lemma~\ref{semisimple-runs}, each repeated-fibre position contributes
  $x/(1-x)$ and each singleton-fibre position contributes $y$. The
  constant semisimple words of length at least three contribute
  $x^3/(1-x)$. This proves the first identity. Substitute it into
  Theorem~\ref{bivariate} to obtain the second. Equivalently, use
  Theorem~\ref{module-substitution-thm}: its quotient $11$ accounts for
  all constant words of length at least two, and each repeated position
  of a longer simple quotient may carry any nonempty constant run.
\end{proof}

The numbers $s_n$ begin
\begin{equation}\label{eq:prim-counts}
  (s_n)_{n\geq0}
  =1,1,3,2,10,70,634,6742,82306,1126846,17050626,282097790,\ldots
\end{equation}
agreeing with the number of associated structures $R(\omega)$ that are
indecomposable at every length.
We have found neither the sequence~\eqref{eq:prim-counts} in
the OEIS~\cite{oeis} nor a closed form
more explicit than the generating-function identity in
Theorem~\ref{primitive-bivariate}.

\section{Restricted growth functions}\label{sec:rgf}

A nonempty Cayley permutation $\omega=w_1\cdots w_n$ is a \emph{restricted growth
function}, or RGF, if $w_1=1$ and $w_{i+1}\leq\max(w_1,\ldots,w_i)+1$
for $1\leq i<n$. The empty word is also an RGF.
Equivalently, in the ballot of $\omega$ the blocks
appear in increasing order of their minima. Thus, the bijection between Cayley
permutations and ballots restricts to a bijection between RGFs and set
partitions.
We inherit the definition of semisimplicity from Cayley permutations.  The
enumeration of semisimple RGFs requires only minor changes: the blocks are
unordered, and no RGF is a nontrivial skew sum, since every nonempty RGF
starts with $1$.

There is also an intrinsic version of the relational structure of
Section~\ref{sec:modules} for set partitions.  If $\rho$ is a set partition of
$[n]$, let
\[
  R(\rho)=\bigl([n],<,\mathord\sim_\rho\bigr),
\]
where $<$ is the natural order and $i\sim_\rho j$ if $i$ and $j$ belong to the
same block of $\rho$.

\begin{proposition}\label{rgf-modules}
  Let $\omega$ be an RGF and $\rho$ its corresponding set partition. A nonempty
  positional interval is an H-interval of $\omega$ if and only if it is a union
  of blocks of $\rho$. Moreover, $R(\rho)$ and $R(\omega)$ have the same modules.
  Consequently,
  $R(\rho)$ is indecomposable if and only if $\omega$ is simple.
\end{proposition}

\begin{proof}
  Let $I$ be a nonempty positional interval that is a union of blocks.
  If blocks $A,B,C$ satisfy $A,C\subseteq I$ and $\min A<\min B<\min C$,
  then $\min B\in I$, hence $B\subseteq I$. The blocks in $I$ are consecutive in their
  order by minima, so $I$ is an H-interval by Proposition~\ref{ballot}.
  The converse follows from the same proposition.

  The natural order forces every nonempty module $M$ of $R(\rho)$ to be
  a positional interval. If $M$ meets at least two blocks, the relation
  $\sim_\rho$ forces every block meeting $M$ to be contained in $M$, so the first
  assertion makes $M$ an H-interval. Otherwise $\omega$ is constant on $M$.
  Conversely, both kinds of interval are modules of $R(\rho)$. By
  Lemma~\ref{modules}, these are exactly the nonempty modules of $R(\omega)$;
  the empty set is a module of both structures. The final assertion now
  follows from Theorem~\ref{indecomposable}.
\end{proof}

\begin{lemma}\label{rgf-inflation}
  Assuming $\sigma\in\Cay_k$ and
  $\alpha_1,\ldots,\alpha_k\in\Cay$ are nonempty,
  an inflation $\sigma[\alpha_1,\ldots,\alpha_k]$ is an RGF if and
  only if $\sigma,\alpha_1,\ldots,\alpha_k$ are all RGFs.
\end{lemma}

\begin{proof}
  A Cayley permutation is an RGF precisely when, for each value $v$ below its
  maximum, the first occurrence of $v$ precedes the first occurrence of $v+1$.
  In $\omega=\sigma[\alpha_1,\ldots,\alpha_k]$ the values split into consecutive
  bands indexed, in increasing order, by the values of $\sigma$.  Thus the RGF
  condition splits into conditions within the bands and between consecutive
  bands.  The band indexed by $j$ is a single value if the fibre
  $\sigma^{-1}(j)$ is not a singleton, in which case the corresponding
  components are constant. If $\sigma^{-1}(j)=\{i\}$, the values in
  the band first occur in the order in which they first occur in $\alpha_i$.
  The conditions within the bands therefore say precisely that every
  $\alpha_i$ is an RGF.  Suppose they hold.  For either kind of fibre,
  the largest value of the $j$th band first occurs in the part indexed by
  $\min\sigma^{-1}(j)$.  Since every $\alpha_i$ begins with $1$, the smallest
  value of the $(j+1)$st band first occurs at the first position of the part
  indexed by $\min\sigma^{-1}(j+1)$.  The parts are disjoint intervals occurring
  in index order, so the conditions between consecutive bands say precisely that
  $\min\sigma^{-1}(j)<\min\sigma^{-1}(j+1)$ for every $j$, which is to say that
  $\sigma$ is an RGF.
\end{proof}

\begin{theorem}\label{rgf-enumeration}
  Let $B_n^{\rgf}(t)=\sum_{\omega}t^{\mu(\omega)}$, the sum being over the RGFs $\omega$ of
  length $n$, let $B^{\rgf}(x,t)=\sum_{n\geq1}B_n^{\rgf}(t)x^n$, and let
  $\widetilde S_{\geq3}^{\rgf}$ and $\widetilde S^{\rgf}$ be defined as $\widetilde S_{\geq3}$ and $\widetilde S$ with
  ``semisimple'' replaced by ``semisimple RGF''.  Then
  \[
    1+\sum_{n\geq1}B_n^{\rgf}(t)\frac{x^n}{n!}
    =\exp\bigl(tx+e^x-1-x\bigr),
  \]
  and, for $n\geq1$,
  \[
    B_n^{\rgf}(t)
    =\sum_{k=1}^{n}\sum_{j=0}^{k}
     \binom{n}{j}\stirling{n-j}{k-j}(t-1)^j.
  \]
  Moreover, for $n\geq1$ and $0\leq j\leq n$,
  \[
    [t^j]\,B_n^{\rgf}(t)
    =\binom{n}{j}\sum_{i\geq0}\stirling{n-j}{i}_{\!\geq2}.
  \]
  Furthermore,
  \begin{align*}
    \widetilde S_{\geq3}^{\rgf}\bigl(x,B^{\rgf}(x,t)\bigr)
    &= B^{\rgf}(x,t)-tx-x^2
      -\frac{B^{\rgf}(x,t)^2}{1+B^{\rgf}(x,t)},
    \intertext{and, letting $U^{\rgf}(x)$ be the unique formal power series with
      $B^{\rgf}\bigl(x,U^{\rgf}(x)\bigr)=x$,}
    \widetilde S^{\rgf}(x)
    &= 2x+\frac{x^3}{1+x}-x\mskip1mu U^{\rgf}(x).
  \end{align*}
\end{theorem}

\begin{proof}
  The blocks of a set partition are unordered, so the sequence of blocks in the
  proof of Proposition~\ref{B-egf} is replaced by a set of blocks and $1/(1-u)$
  becomes $e^u$; in the two formulas of Proposition~\ref{B-explicit} the same
  change removes the factors $k!$ and $(i+j)!$.

  Lemma~\ref{rgf-inflation} verifies~\eqref{eq:inflation-compatible} for the
  RGF property, which holds for $1$, $11$, and $12$, but not for $21$. The
  remaining identities and the existence and uniqueness of $U^{\rgf}$ follow
  from Lemma~\ref{enumeration-compatible},
  Lemma~\ref{U-lemma}, and Corollary~\ref{univariate-compatible}, with
  $F_Q=B^{\rgf}$ and $V_Q=U^{\rgf}$.
\end{proof}

Writing $u_n^{\rgf}=[x^n]U^{\rgf}(x)$, the coefficients
$\tilde s_n^{\rgf}=[x^n]\widetilde S^{\rgf}(x)$, for $n\geq1$, satisfy
\[
  \tilde s_n^{\rgf}=-u_{n-1}^{\rgf}+(-1)^{n-1}
\]
for $n\geq3$. Including $\tilde s_0^{\rgf}=1$, the sequence begins
\[
  (\tilde s_n^{\rgf})_{n\geq0}
  =1,1,2,2,4,13,51,228,1129,6093,35351,218467,\ldots
\]
and it is not in the OEIS~\cite{oeis}.

For simple RGFs, we use the same method as for the coefficients
$s_n$ in Theorem~\ref{primitive-bivariate}. Define
\[
  P^{\rgf}(x,t)
  =\sum_{\substack{\omega\in\Cay_+\\\omega\text{ primitive RGF}}}
     x^{|\omega|}t^{\mu(\omega)},
\]
and let $W^{\rgf}$ satisfy $P^{\rgf}(x,W^{\rgf}(x))=x$. Run expansion and
Corollary~\ref{univariate-compatible}, adding $x^2$ for $11$, give
\begin{align*}
  P^{\rgf}(x,t)
    &=B^{\rgf}\!\left(\frac{x}{1+x},\,(1+x)t-x\right),\\
  S^{\rgf}(x)
    &=2x+x^2+\frac{x^3}{1+x}-xW^{\rgf}(x).
\end{align*}
Lemma~\ref{semisimple-runs} preserves the RGF property, so the first
identity of Corollary~\ref{semisimple-simple-series} also holds with RGF
superscripts. Its second identity holds with $B^{\rgf}$ and
$S_{\geq3}^{\rgf}$, and with coefficient $1$ instead of $2$
on the sum term, since no RGF is a nontrivial skew sum.

Thus, writing $s_n^{\rgf}=[x^n]S^{\rgf}(x)$ for $n\geq1$
and including $s_0^{\rgf}=1$, we obtain
\[
  (s_n^{\rgf})_{n\geq0}=1,1,2,1,1,5,16,69,316,1591,8614,\ldots
\]
This sequence is related to OEIS A099947, the numbers $t_n$ of topologically
connected set partitions~\cite{oeis}. These partitions have no proper
subinterval that is a union of blocks; Beissinger~\cite{Bei85} calls them
irreducible, and Dykema~\cite{Dyk16} calls them
connected. By Proposition~\ref{rgf-modules},
for $n\geq2$ these partitions correspond to semisimple RGFs with no singleton
fibres. Hence their generating function is
\[
  T(x)=\sum_{n\geq1}t_nx^n=x+x^2+\widetilde S_{\geq3}^{\rgf}(x,0).
\]
Combining run expansion with the deletion of marked singleton blocks gives
\[
  P^{\rgf}(x,t)
  =x+(1+x)B^{\rgf}\!\left(x,(1+x)(t-1)\right).
\]
Substituting $t=W^{\rgf}(x)$ in this identity and using
Theorem~\ref{rgf-enumeration} gives
\[
  T(x)=(1+x)S^{\rgf}(x)-2x^2-2x^3.
\]
Thus $t_n=s_n^{\rgf}+s_{n-1}^{\rgf}$ for $n\geq4$.

\section*{Acknowledgements}

We thank Christian Bean and Robert Brignall for valuable feedback on an
earlier version on this article.

The large language models GPT (OpenAI) and Claude (Anthropic)
were used during the preparation of this manuscript to assist with exploring and
checking conjectures through programming, drafting some propositions
and their proofs, light editing, and proofreading.

\end{document}